\documentclass[twocolumn]{article}

\usepackage{microtype}
\usepackage{graphicx}
\usepackage{subcaption}
\usepackage{booktabs}
\usepackage{stmaryrd}
\usepackage{amsmath}
\usepackage{amssymb}
\usepackage{mathtools}
\usepackage{amsthm}
\usepackage{hyperref}
\usepackage[capitalize,noabbrev]{cleveref}
\usepackage{xcolor}
\usepackage{listings}
\usepackage{natbib}

\theoremstyle{plain}
\newtheorem{theorem}{Theorem}[section]

\newtheorem{lemma}[theorem]{Lemma}

\theoremstyle{definition}

\theoremstyle{remark}
\newtheorem{remark}[theorem]{Remark}

\definecolor{codegreen}{rgb}{0,0.6,0}
\definecolor{codegray}{rgb}{0.5,0.5,0.5}
\definecolor{codepurple}{rgb}{0.58,0,0.82}
\definecolor{backcolour}{rgb}{0.95,0.95,0.92}
\definecolor{bluekeywords}{rgb}{0,0,1}

\lstdefinelanguage{lean}{morekeywords={def, theorem, lemma, example, definition, structure, instance,
		inductive, match, with, where, let, in, if, then, else,
		forall, fun, $\lambda$, have, show, from, by, begin, end,
		import, open, namespace, section, variable, variables,
		noncomputable, @[simp], @[ext], @[class], @[instance],
		deriving, extends, of, sorry, axiom, constant,
		protected, private, mutual, abstract},
	sensitive=true,
	morecomment=[l]{--},
	morecomment=[s]{/-}{-/},
	morestring=[b]",
	morestring=[b]',
	morestring=[b]"""}  
\lstdefinestyle{leanstyle}{backgroundcolor=\color{backcolour},
	frame=single,
	rulecolor=\color{black},
	commentstyle=\color{codegreen},
	keywordstyle=\color{bluekeywords},
	numberstyle= \tiny\color{codegray},
	stringstyle=\color{codepurple},
	basicstyle=\ttfamily\footnotesize,
	breakatwhitespace=false,
	breaklines=true,
	captionpos=b,
	keepspaces=true,
	numbers=none,
	numbersep=5pt,
	showspaces=false,
	showstringspaces=false,
	showtabs=false,
	tabsize=2,
	alsoletter={_},
	columns=fullflexible,
	mathescape=true, 
}
\title{Formalizing Extended Complex Numbers, M\"obius Transformations, \\
	and Cross Ratio in Lean 4}

\author{Fubin YAN \\
	School of Science and Engineering, \\
	The Chinese University of Hong Kong, Shenzhen \\
	\texttt{fubinyan@link.cuhk.edu.cn}
	\and
	Kenneth SHUM \\
	School of Science and Engineering, \\
	The Chinese University of Hong Kong, Shenzhen \\
	\texttt{wkshum@cuhk.edu.cn}}

\begin{document}

\maketitle

\begin{abstract}
	The extended complex plane is a fundamental object in complex analysis, hyperbolic geometry, and mathematical physics. Its geometry is governed by M\"obius transformations, with the cross ratio serving as a central invariant. We present a formalization of these concepts in the Lean~4 theorem prover. The extended complex plane is represented using Mathlib's Option type over $\mathbb{C}$, where the additional element represents the point at infinity. On this foundation, we define M\"obius transformations, their action on the extended complex plane, and the cross ratio. We formalize several basic properties of M\"obius transformations, including their group structure, and identify them with a projective general linear group. We also prove the uniqueness of a M\"obius transformation mapping any three distinct points to any other three distinct points, and the invariance of the cross ratio. The complete development comprises approximately 6,000 lines of Lean code, including about 40 definitions and 150 lemmas and theorems. This work provides a verified foundation for future formalizations of conformal geometry, hyperbolic models, modular forms, and applications in mathematical physics.
\end{abstract}

\section{Introduction}

Complex analysis is a fundamental area of mathematics, with numerous applications in both pure mathematics and engineering. Standard textbooks on complex analysis, such as~\cite{Ahlfors, BrownChurchill, SteinShakarchi}, often emphasize the rich geometric intuition underlying the subject. However, the formalization of complex analysis in an interactive theorem prover can look quite different from its usual presentation in the classroom. In this paper, we extend the complex plane by adding a point at infinity, and investigate an automorphism group using the Lean 4 proof assistant.

We denote the \emph{extended complex plane} by
\[
\mathbb{C}_\infty := \mathbb{C} \cup \{\infty\}.
\]
The additional point $\infty$ is called the \emph{point at infinity}. Through stereographic projection, the points of the complex plane are in one-to-one correspondence with the points of the unit sphere in three-dimensional space, except for the north pole, which corresponds to the point at infinity. For this reason, the extended complex plane is also known as the \emph{Riemann sphere}, which is a one-point compatification of the complex plane, and is the first nontrivial example of compact Riemann surface~\cite{FarkasKra}.

With the point at infinity adjoined, the geometric model of the Riemann sphere unifies lines and circles: under stereographic projection, every line or circle in the complex plane is mapped to a circle on the sphere. In particular, a line in the complex plane corresponds to a circle on the Riemann sphere passing through the point at infinity. This circle-line geometry is naturally preserved
by M\"obius transformations, which we define below.

A \emph{M\"obius transformation}, also called a \emph{linear fractional transformation}, is a function of the form
\[
f(z) = \frac{az+b}{cz+d},
\]
where $a$, $b$, $c$, and $d$ are complex numbers satisfying $$ad-bc \neq 0.$$ 
The four coefficients of a M\"obius transformation can be arranged into the
coefficient matrix
\begin{equation}
	\begin{pmatrix}
		a & b \\
		c & d
	\end{pmatrix}.
	\label{eq:coefficient-matrix}
\end{equation}
Its determinant is
$ad-bc$.
The non-degeneracy condition $ad-bc \neq 0$ is precisely the condition that
this matrix is invertible.

To regard such a transformation as a total function from $\mathbb{C}_\infty$ to
$\mathbb{C}_\infty$, we extend its definition at the point where the denominator vanishes and at infinity. When $c \neq 0$, we set
\[
f(-d/c)=\infty
\qquad\text{and}\qquad
f(\infty)=a/c.
\]
When $c=0$, the point at infinity is mapped to itself.

We note that multiplying all four coefficients $a,b,c,d$ by the same nonzero complex number gives the same transformation. Thus a M\"obius transformation depends only on the projective equivalence class of its coefficient matrix.

The class of M\"obius transformations includes several important special cases:
\begin{enumerate}
	\item translations $f(z)=z+b$, where $b \in \mathbb{C}$ is a complex constant;
	\item dilations $f(z)=az$, where $a>0$ is a positive real constant;
	\item rotations $f(z)=\omega z$, where $|\omega|=1$;
	\item inversion $f(z)=1/z$.
\end{enumerate}
For translations, dilations, and rotations, the point at infinity is fixed. In
contrast, inversion interchanges $0$ and $\infty$.

As maps from $\mathbb{C}_\infty$ to itself, M\"obius transformations form a group under composition. They map generalized circles to generalized circles, where generalized circles are the ordinary circles and lines in the complex
plane. They also preserve the \emph{cross ratio}, a fundamental invariant in projective geometry and complex analysis. M\"obius transformations appear in many areas of mathematics and physics, including complex analysis, hyperbolic geometry and the Poincar\'e disk model, conformal field theory, and special relativity.

Although the formalization of complex analysis in interactive theorem provers has made significant progress, the extended complex plane and the M\"obius group have not yet been extensively developed in Lean 4. This paper takes an initial
step toward filling this gap.

The main constructions and results of this paper are as follows.

\begin{itemize}
	\item We define the extended complex numbers $\mathbb{C}_\infty$ using Lean's \texttt{Option} type, and develop arithmetic lemmas for extended complex numbers and linear fractional transformations. These lemmas cover addition, multiplication, inversion, and the treatment of the point at infinity.
	
	\item We formalize linear fractional transformations, which are  also known as M\"obius transformations, and prove the basic algebraic properties for composition, inverses, and identity, thereby establishing that they form a group, as an instance of the \texttt{Group} typeclass.
	
	\item We prove that two M\"obius transformations act identically on all extended complex numbers if and only if their coefficient matrices are proportional. We also prove that there exists a \emph{unique} M\"obius transformation mapping any three distinct points to any other three distinct points, and provide an explicit construction.

	\item We construct the quotient group \(\operatorname{PSL}(2,\mathbb{C})\), which is the group of linear fractional transformations modulo proportionality, as a formal instance of the \texttt{Group} typeclass. We also construct a subgroup isomorphic to $S_4$, the symmetric group on four elements.
	
	\item We define the cross ratio in a way that covers all cases in which some of the points may be equal to infinity, and prove its invariance under M\"obius transformations.  
\end{itemize}

The rest of the paper is organized as follows. Section \ref{sec:SOTA} reviews existing works on the formalization of complex analysis. Section~\ref{sec:EComplex} presents the extended complex plane \texttt{EComplex}. Section~\ref{sec:LFT} introduces M\"obius transformations, their group structure, their action, and the quotient by scalar multiplication forming $\operatorname{PGL}(2,\mathbb{C})$. Section~\ref{sec:mapping_property} proves a key mapping property of M\"obius transformation. In Section~\ref{sec:cross_ratio}, we define the cross ratio, prove its M\"obius invariance, and state the permutation generator lemmas. Section \ref{sec:discussion} discusses the algebraic methodology and its relation to AI-assisted formalization. We conclude in Section~\ref{sec:conclusion} and outline the possible future work.

\section{Existing works on the formalization of complex analysis}
\label{sec:SOTA}

Several formalization projects in Isabelle/HOL, HOL Light, HOL4, and Lean have developed substantial libraries for complex analysis. We review some of the most relevant work in this section.

\subsection{HOL Light, HOL4, and Isabelle/HOL}

Harrison developed an extensive complex analysis library in HOL Light, including the Cauchy integral formula~\cite{Harrison2007}. Building on this library, he later formalized an analytic proof of the prime number theorem~\cite{Harrison2009}. Cauchy's residue theorem and the argument principle were also formalized in Isabelle/HOL by Li and Paulson~\cite{LiPaulson2016}.

Shi, Guan, and Li wrote a book on the formalization of complex analysis in HOL4, with particular emphasis on applications to the discrete Fourier transform~\cite{Shi2020}.

Mari\'c and Petrovi\'c formalized complex plane geometry in Isabelle/HOL, including M\"obius transformations, circles, and the cross ratio~\cite{maric2014formalizing}. Their development uses homogeneous coordinates and quotient types. In particular, they formalized the group action of M\"obius transformations on circles and straight lines, and the stereographic projection.

\subsection{Lean}

The mathematical library of Lean 4 contains a substantial development of complex analysis on the standard complex plane $\mathbb{C}$~\cite{Lean}. It includes Cauchy's integral theorem for rectangles and circles,  Cauchy integral formula for circles, and the Liouville's theorem. However, notions such as M\"obius transformations  on $\mathbb{C}_\infty$ and the cross ratio involving points at infinity are not yet available in Mathlib. An independent line of work~\cite{ray_Riemann_sphere} implements the Riemann sphere from the perspective of complex manifolds, with emphasis on topological properties of Riemann sphere. In contrast, the work in this paper is more algebraic.

\paragraph{The strong prime number theorem.}
Led by Tao and Kontorovich, the PNT+ project aims to formalize the prime number theorem and its refinements using analytic methods~\cite{KontorovichTao}. The project relies heavily on contour integration, but it works entirely over $\mathbb{C}$ and does not extend the complex plane by adding a point at infinity. An auto-formalization of the strong prime number theorem was carried out in~\cite{StrongPNT}.

\paragraph{The $E_8$ sphere packing formalization.}
Viazovska's work on the sphere packing problem in dimension $8$ was formalized in Lean~\cite{Hariharan2026}. A key component of the proof is modular forms, which are functions on the upper half-plane satisfying transformation laws under the modular group $PSL(2,\mathbb{Z})$. These transformation laws are special cases of M\"obius transformations, and the behaviour of modular forms at the cusp corresponds to their behaviour at the point at infinity. Therefore, a formalization of M\"obius transformations and $\mathbb{C}_\infty$ may provide a useful foundation for future developments on modular forms, automorphic forms, and their applications to number theory and geometry.

Among the works discussed above, the formalization by Mari\'c and Petrovi\'c is closest to the development presented in this paper. Nevertheless, our work differs from theirs in several respects. First, our development is carried out in Lean 4 and makes use of its typeclass mechanism and the existing Mathlib infrastructure. Second, we adopt a simpler representation of the extended complex plane $\mathbb{C}_\infty$, namely as \texttt{Option} $\mathbb{C}$, rather than using homogeneous coordinates. This design choice makes equality of points more straightforward to check, but it also means that some proofs require explicit case distinctions. Such details are usually omitted in informal mathematical presentations; in this work, we make them fully explicit.

\section{The Extended Complex Plane \texttt{EComplex}}
\label{sec:EComplex}
In our formalization, the extended complex plane is represented by the \texttt{Option} type over $\mathbb{C}$. Thus, elements of \texttt{EComplex} are either ordinary complex numbers, represented by \texttt{some z}, or the additional element \texttt{none}. We interpret \texttt{none} as the point at infinity and denote it by the customary symbol~$\infty$.
	
\begin{lstlisting}[language=lean,mathescape=true]
def EComplex := Option $\mathbb{C}$
  notation "$\infty$" => none
\end{lstlisting}

This design choice follows the standard presentation of the extended complex
plane in textbooks such as~\cite[Section 99]{BrownChurchill} and
\cite[Section 2.4]{Ahlfors}. In this presentation, one starts with the complex
plane and adjoins a single extra point, denoted by $\infty$. Arithmetic and
M\"obius transformations are then defined by extending the usual operations on
$\mathbb{C}$ with explicit rules for the point at infinity.

An alternative formalization is to use projective coordinates, as in
\cite{maric2014formalizing}. In that approach, a point of the Riemann sphere is
represented by an equivalence class of nonzero pairs of complex numbers. This
representation is well suited to some algebraic arguments, since M\"obius
transformations can be treated directly as matrix actions on homogeneous
coordinates. However, equality of points is then equality modulo scalar
multiplication, which makes the representation less direct for statements that
refer explicitly to ordinary complex numbers and the distinguished point
$\infty$.

In this paper, we choose the \texttt{Option} representation because it matches
the textbook viewpoint most closely: the Riemann sphere is treated as the
disjoint union of the complex plane and one additional point. This choice makes
the embedding of $\mathbb{C}$ and the special role of $\infty$ transparent,
although it also requires us to handle the exceptional cases involving
$\infty$ explicitly in proofs. Our development therefore follows the usual
textbook proof strategy, while making all such cases precise in Lean.

	
Arithmetic is defined by pattern matching. For example, we define addition by
	
\begin{lstlisting}[language=lean,mathescape=true]
def add : EComplex $\to$ EComplex $\to$ EComplex
  | $\infty$, $\infty$ => $\infty$
  | (z : $\mathbb{C}$), $\infty$ => $\infty$
  | $\infty$, (z : $\mathbb{C}$) => $\infty$
  | (z : $\mathbb{C}$), (w : $\mathbb{C}$) => (z + w : $\mathbb{C}$)
\end{lstlisting}
When either $z$ or $w$ is equal to $\infty$, we define $z+w$ as $\infty$. Otherwise, $z+w$ is the usual complex addition.

Multiplication, subtraction, division and inversion are defined similarly, with conventions that $0 \cdot \infty = \infty$ and $\infty \cdot \infty = \infty$, consistent with the topology of the Riemann sphere. The function 

\begin{lstlisting}[language=lean,mathescape=true]
toComplex : EComplex $\to \mathbb{C}$
\end{lstlisting}
coerce a number in the extended complex plane back to $\mathbb{C}$. We make a choice of mapping  infinity to $0$ (a convenient choice for algebraic manipulation).

\begin{remark} Mathlib already contains a general library for one-point compactifications,
	namely
	\[
	\textsf{Mathlib.Topology.Compactification.OnePoint.Basic}.
	\]
	Although this library also adjoins a new point using the \texttt{Option} type, its
	purpose is primarily topological. It provides the general construction of the one-point compactification and supports reasoning about compactness, neighborhoods, and convergence. Our development has a different goal. We formalize the algebraic and geometric operations on $\mathbb{C}_\infty$.
\end{remark}

\section{M\"obius Transformations}
\label{sec:LFT}

\subsection{Definition and Action}

A M\"obius transformation is represented by a structure containing the four coefficients and a proof of non-degeneracy. In the Lean code, we use the name \texttt{LinearFracTrans}, reflecting the equivalent term \emph{linear fractional transformation} commonly used for M\"obius transformations.

\begin{lstlisting}[language=lean,mathescape=true]
structure LinearFracTrans where
  a : $\mathbb{C}$
  b : $\mathbb{C}$
  c : $\mathbb{C}$
  d : $\mathbb{C}$
  determinant_ne_zero : a * d - b * c $\neq$ 0
\end{lstlisting}
	
Its action on an extended complex number is defined by case analysis on whether $c=0$ and whether the input is the pole $-d/c$. We define a function $f(z) = (az+b)/(cz+d)$ on the extended complex plane. 

\noindent When $c=0$, 
$$
f(z) \triangleq  \begin{cases}
	(a/d)z + (b/d) & \text{ if } z \neq \infty,\\
	\infty & \text{ if } z = \infty.
\end{cases}
$$
When $c\neq 0$, 
$$
f(z) \triangleq  \begin{cases} 
	\frac{az+b}{cz+d} & \text{ if } z\neq -d/c,\ z \neq \infty\\
	\infty & \text{ if } z = -d/c \\
	a/c & \text{ if } z = \infty.
\end{cases}
$$
The implementation in Lean is as follows:

\newpage 
	
\begin{lstlisting}[language=lean,mathescape=true]
def apply (f : LinearFracTrans) (z : EComplex) : EComplex :=
  if f.c = 0 then
    match z with
    | some z => some ((f.a / f.d) * z + (f.b / f.d))
    | $\infty$ => $\infty$
  else
    match z with
    | some z =>
      if z = -f.d / f.c then $\infty$
      else some ((f.a * z + f.b) / (f.c * z + f.d))
    | $\infty$ => some (f.a / f.c)
\end{lstlisting}
	
We also equip \texttt{LinearFracTrans} with a \texttt{CoeFun} instance so that $f\,z$ can be written instead of \texttt{apply f z}.
	
\subsection{Group Structure and Group Action}
	
Composition is defined as matrix multiplication. In addition to the computation of the parameters $a$, $b$, $c$, and $d$ of the composed function $f\circ g$, we also need to provide a proof that the determinant of the composed function is nonzero.
	
\begin{lstlisting}[language=lean,mathescape=true]
def comp (f g : LinearFracTrans) : LinearFracTrans where
  a := f.a * g.a + f.b * g.c
  b := f.a * g.b + f.b * g.d
  c := f.c * g.a + f.d * g.c
  d := f.c * g.b + f.d * g.d
  determinant_ne_zero := by
    have h3 : (f.a * g.a + f.b * g.c) * (f.c * g.b + f.d * g.d) - (f.a * g.b + f.b * g.d) * (f.c * g.a + f.d * g.c) =
              (f.a * f.d - f.b * f.c) * (g.a * g.d - g.b * g.c) 
      := by ring
    rw [h3]; 
    exact mul_ne_zero f.determinant_ne_zero g.determinant_ne_zero
\end{lstlisting}
	
The identity is represented by \texttt{id} with coefficients $(1,0,0,1)$, and the inverse \texttt{inv} is given by the adjugate matrix divided by the determinant. We prove the necessary algebraic lemmas:
\begin{itemize}
	\item \texttt{lft\_mul\_assoc}: associativity of composition,
	\item \texttt{lft\_one\_mul} and \texttt{lft\_mul\_one}: identity laws,
	\item \texttt{lft\_mul\_left\_inv} (and \texttt{lft\_mul\_right\_inv}): inverse laws.
\end{itemize}

With these lemmas, we instantiate the \texttt{Group} typeclass for \texttt{LinearFracTrans}:

\begin{lstlisting}[language=lean,mathescape=true]
instance : Group (LinearFracTrans) where
  mul := comp
  mul_assoc := by
    intro f g h
    dsimp [$\cdot * \cdot$]
    exact lft_mul_assoc f g h
  one := LinearFracTrans.id
  one_mul := by
    dsimp [$\cdot * \cdot$]
    apply lft_one_mul
  mul_one := by
    dsimp [$\cdot * \cdot$]
    apply lft_mul_one
  inv := LinearFracTrans.inv
  inv_mul_cancel := by
    dsimp [$\cdot * \cdot$]
    apply lft_mul_left_inv
\end{lstlisting}
	
The group operation \texttt{mul} (denoted \texttt{·*·}) is exactly \texttt{comp}.

To turn this group into a group action on \texttt{EComplex}, we prove the fundamental compatibility lemma:
	
\begin{lstlisting}[language=lean,mathescape=true]
theorem comp_equivalent (z : EComplex) (f g : LinearFracTrans) :
  (comp f g) z = f (g z)
\end{lstlisting}
	
The proof is a large case distinction (over whether $f.c = 0$, $g.c = 0$, and whether $z$ is a pole), but each subcase is resolved by algebraic simplification using \texttt{field\_simp} and \texttt{ring}. With this lemma, we register \texttt{LinearFracTrans} as a \texttt{MulAction} on \texttt{EComplex}:

\begin{lstlisting}[language=lean,mathescape=true]
instance : MulAction LinearFracTrans EComplex where
  one_smul := id_apply
  mul_smul f g z := comp_equivalent z f g
\end{lstlisting}
	
Now we can write $f \bullet z$ (or simply $f\,z$ thanks to the \texttt{CoeFun} instance) and use the standard group action theorems from Mathlib.

\subsection{Quotient Group: From Linear Fractional Transformations to \(\operatorname{PSL}(2,\mathbb{C})\)}
	
	
A single M\"obius transformation is not represented by a unique quadruple of coefficients. Indeed, if all four coefficients are multiplied by the same nonzero complex scalar, then the resulting fractional expression defines the same function on the extended complex plane,
\[
\frac{az+b}{cz+d}
=
\frac{(ka)z+kb}{(kc)z+kd}
\qquad (k \in \mathbb{C},\ k \neq 0).
\]
Thus, the structure \texttt{LinearFracTrans} should be understood as a
representative of a M\"obius transformation rather than as the transformation
itself. To identify different representatives that define the same map, we
introduce a proportionality relation on \texttt{LinearFracTrans}. Two
representatives are proportional if their four coefficients differ by
multiplication by a common nonzero scalar.
	
\begin{lstlisting}[language=lean,mathescape=true]
def Proportional (f1 f2 : LinearFracTrans) : Prop :=
  $\exists$ (k : $\mathbb{C}$), k $\neq$ 0 $\land$
    f2.a = k * f1.a $\land$
    f2.b = k * f1.b $\land$
    f2.c = k * f1.c $\land$
    f2.d = k * f1.d
\end{lstlisting}
This relation identifies exactly those coefficient quadruples that differ by a common nonzero scalar factor, and hence represent the same projective matrix. The quotient by this relation is the formal object corresponding to the M\"obius group.
	
A key result connects functional equality with proportionality:

\begin{lstlisting}[language=lean,,mathescape=true]
theorem funEq_iff_proportional (f1 f2 : LinearFracTrans) :
  FunEq f1 f2 $\leftrightarrow$ Proportional f1 f2
\end{lstlisting}

Thus, two transformations that are proportional give the same function on \(\mathbb{C}_\infty\). Using this, we define an equivalence relation \texttt{lftEquiv} (proportionality) and form the quotient type:
	
\begin{lstlisting}[language=lean]
def lftsetoid : Setoid LinearFracTrans := { r := lftEquiv
  iseqv := {
    refl := lftproportional_refl
    symm := lftproportional_symm
    trans := lftproportional_trans
  }}   -- equivalence by proportionality
\end{lstlisting}
	
\begin{lstlisting}[language=lean,mathescape=true]
def quotMul : Quotient lftsetoid $\to$ Quotient lftsetoid $\to$ Quotient lftsetoid := ...
def quotInv : Quotient lftsetoid $\to$ Quotient lftsetoid := ...
def quotOne : Quotient lftsetoid := $\llbracket$1$\rrbracket$
\end{lstlisting}
	



The quotient in the Lean code reflects the fact that a M\"obius transformation
is represented by a projective equivalence class of
invertible matrices. Indeed, the coefficients $(a,b,c,d)$ form the matrix
\[
\begin{pmatrix}
	a & b \\
	c & d
\end{pmatrix},
\]
whose determinant is $ad-bc$. The condition of nonvanishing determinant says that this matrix
lies in $\operatorname{GL}(2,\mathbb{C})$. Multiplying the matrix by a nonzero
scalar does not change the induced map on $\mathbb{C}_\infty$, so the actual
group of transformations is
\[
\operatorname{GL}(2,\mathbb{C})/\mathbb{C}^{\times}
=
\operatorname{PGL}(2,\mathbb{C}).
\]
This is implemented in Lean by the quotient \texttt{Quotient lftsetoid}:
\begin{lstlisting}[language=lean]
instance : Group (Quotient lftsetoid) := { ... }
\end{lstlisting}
Over $\mathbb{C}$, this group is naturally isomorphic to
$\operatorname{PSL}(2,\mathbb{C})$, and it is the full automorphism group of the
Riemann sphere.
	

\section{Existence and Uniqueness Determined by Three Points}
\label{sec:mapping_property}
A classic theorem in complex variables states that a M\"obius transformation is uniquely determined by its images of three distinct points. We prove this in two steps.

\subsection{Fixing Three Points Implies Identity}

First, we show that if a transformation fixes three distinct points, it must be the identity.

\begin{lemma}
	If $z_1$, $z_2$ and $z_3$ are there distinct points in the extended complex plane, and $f(z)$ is a linear fractional transformation satisfying $f(z_j)=z_j$, for $j= 1,2,3$, then 
	$$f(z) = z$$
	for all $z\in{\mathbb{C}_\infty}$.
	\label{lemma:uniqueness}
\end{lemma}
	
This is formulated in Lean as follows.

\begin{lstlisting}[language=lean,mathescape=true]
lemma fixes_three_is_id 
  (f : LinearFracTrans) 
  (z1 z2 z3 : EComplex)
  (hz : z1 $\neq$ z2 $\land$ z2 $\neq$ z3 $\land$ z3 $\neq$ z1)
  (hfix1 : f z1 = z1) 
  (hfix2 : f z2 = z2) 
  (hfix3 : f z3 = z3) :
    FunEq f id
\end{lstlisting}
	
The proof distinguishes whether $\infty$ is among the fixed points. If $\infty$ is fixed, then $c=0$ and the transformation becomes affine $z \mapsto (a/d)z + b/d$. Using the two finite fixed points, we deduce $a/d = 1$ and $b=0$. If all three fixed points are finite, we derive quadratic equations that force $c=0$ (by showing a non-zero $c$ would lead to three distinct roots of a quadratic, impossible), then reduce to the affine case.
	
\subsection{Uniqueness and Existence}

The characterization of linear fractional transformation from action of three distinct points are stated in the theorem below.

\begin{theorem}
	Given any three distinct points $z_1$, $z_2$ and $z_3$ and any three distinct points $w_1$, $w_2$ and $w_3$, all in the extended complex plane ${\mathbb{C}}_\infty$, there is a unique linear fractional transformation that maps $z_1$ to $w_1$, $z_2$ to $w_2$, and $z_3$ to $w_3$. \label{thm:unique_exist}
\end{theorem}

Uniqueness follows immediately from Lemma~\ref{lemma:uniqueness}.
	
\begin{lstlisting}[language=lean,mathescape=true]
theorem lft_uniqueness 
  (f g : LinearFracTrans) 
  (z1 z2 z3 : EComplex)
  (hdist : z1 $\neq$ z2 $\land$ z2 $\neq$ z3 $\land$ z3 $\neq$ z1)
  (h1 : f z1 = g z1) 
  (h2 : f z2 = g z2) 
  (h3 : f z3 = g z3) :
    FunEq f g
\end{lstlisting}
	
\begin{proof} consider $h = g^{-1} \circ f$; it fixes $z_1,z_2,z_3$, hence $h$ is equal to the identity function on the extended complex plane. So $f = g$.
\end{proof}

For existence, we explicitly construct a transformation sending any three points $z_1,z_2,z_3$ to $0,1,\infty$ (by case analysis on which of them is $\infty$), and then compose with the inverse of a similar map for $w_1,w_2,w_3$:

There are three boundary cases to consider. When $z_1=\infty$ and $z_2$ and $z_3$ are finite, we can map $\infty$ to 0, $z_2$ to 1, and $z_3$ to $\infty$ by
\begin{equation}
	f(z) = \frac{z_2-z_3}{z-z_3}.
	\label{eq:LFT2}
\end{equation}
When $z_2=\infty$ and $z_1$ and $z_3$ are finite, we can map $z_1$ to 0, $\infty$ to 1, and $z_3$ to $\infty$ by
\begin{equation}
	f(z) = \frac{z-z_1}{z-z_3}.
	\label{eq:LFT3}
\end{equation}
Finally, when $z_3=\infty$ and $z_1$ and $z_2$ are finite, we can map $z_1$ to 0, $z_2$ to 1, and $\infty$ to $\infty$ by
\begin{equation}
	f(z) = \frac{z-z_1}{z_2-z_1}.
	\label{eq:LFT4}
\end{equation}
	
\begin{lstlisting}[language=lean,mathescape=true]
theorem exist_LFT_mapping_three_point 
  (z1 z2 z3 w1 w2 w3 : EComplex)
  (hz : z1 $\neq$ z2 $\land$ z2 $\neq$ z3 $\land$ z3 $\neq$ z1)
  (hw : w1 $\neq$ w2 $\land$ w2 $\neq$ w3 $\land$ w3 $\neq$ w1) :
  $\exists$ f : LinearFracTrans, f z1 = w1 $\land$ f z2 = w2 
    $\land$ f z3 = w3
\end{lstlisting}
	
Combined with uniqueness, this proves Theorem~\ref{thm:unique_exist}.

\section{Cross Ratio: Definition, Invariance, and Permutation Generators}
\label{sec:cross_ratio}

\subsection{Definition}

Given four  complex numbers $z_0$, $z_1$, $z_2$ and $z_3$, we define the {\em cross ratio} as the image of $z_0$ under the linear transformation that maps $z_1$ to 0, $z_2$ to 1, and $z_3$ to $\infty$. We denote the cross ratio as
$$
[z_0,z_1,z_2,z_3] \triangleq 
\begin{cases}		
	\frac{z_0 - z_1}{z_0-z_3} \cdot \frac{z_2-z_3}{z_2-z_1} & \text{ if } z_0 \neq \infty, \\
	\frac{z_2-z_3}{z_2-z_1} & \text{ if } z_0 = \infty,
\end{cases}
$$
when $z_1$, $z_2$ and $z_3$ are all finite (see p.310 in~\cite{BrownChurchill}). If one of $z_1$, $z_2$ and $z_3$ is $\infty$, the cross ratio is computed by \eqref{eq:LFT2}, \eqref{eq:LFT3} or \eqref{eq:LFT4}.

In Lean, we define the cross ratio of four points by pattern matching on which of them are $\infty$:
\\
\begin{lstlisting}[language=lean,mathescape=true]
def cross_ratio (z0 z1 z2 z3 : EComplex) : EComplex :=
  match z1, z2, z3 with
  | some z1, some z2, some z3 =>
    match z0 with
    | some z0 => (z0 - z1) / (z0 - z3) * ((z2 - z3) / (z2 - z1))
    | none => (z2 - z3) / (z2 - z1)
  | none, some z2, some z3 => (z2 - z3) / (z0 - z3)
  | some z1, none, some z3 => (z0 - z1) / (z0 - z3)
  | some z1, some z2, none => (z0 - z1) / (z2 - z1)
  | _, _, _ => none   -- junk case (never reached under distinctness)
\end{lstlisting}
	
This definition automatically handles the limiting cases when one of the arguments is $\infty$, following the standard conventions.

\subsection{M\"obius Invariance}

The main theorem states that the cross ratio is invariant under any M\"obius transformation.

\begin{theorem}
	Let $z_0$, $z_1$, $z_2$ and $z_3$ be pairwise distinct numbers in $\mathbb{C}_\infty$, and $f(z)$ be a linear fractional transformation mapping $z_k$ to $z_k'$, for $k=0,1,2,3$. Then
	$$
	[z_0,z_1,z_2,z_3] = [z_0', z_1', z_2', z_3'].
	$$
\end{theorem}

This theorem is formulated in Lean as
	
\begin{lstlisting}[language=lean,mathescape=true]
theorem cross_ratio_invariant 
  (f : LinearFracTrans) 
  (z0 z1 z2 z3 : EComplex)
  (h_distinct : List.Pairwise ($\cdot \neq \cdot$) [z0, z1, z2, z3]) :
  cross_ratio (f z0) (f z1) (f z2) (f z3) = cross_ratio z0 z1 z2 z3
\end{lstlisting}
	

The proof of cross-ratio invariance is a good example of the additional precision required by formalization. In textbooks, the invariance of the cross ratio under M\"obius transformations is usually presented as a short algebraic calculation. The exceptional cases, especially those involving the point at infinity or the pole of the transformation, are often treated informally or absorbed into notation. In Lean, these cases cannot be suppressed: the action at
$\infty$, the behavior at the pole, and the possible occurrence of $\infty$ among the four arguments must all be handled explicitly.

Our proof first splits on whether the coefficient $c$ is zero or nonzero. The case $f.c=0$ corresponds to an affine transformation, while $f.c\neq 0$ is the genuinely fractional case. In each branch, we then perform further case analysis according to which of
$z_0,z_1,z_2,z_3$ are equal to $\infty$. Each subcase is reduced to an identity in the complex numbers. The central finite-point calculation is isolated in the following lemma:

\begin{lstlisting}[language=lean,mathescape=true]
lemma cross_ratio_lft_num_denom 
  (f : LinearFracTrans) (z0 z1 z2 z3 : $\mathbb{C}$)
  (h0 : f.c * z0 + f.d $\neq$ 0) (h1 : f.c * z1 + f.d $\neq$ 0)
  (h2 : f.c * z2 + f.d $\neq$ 0) (h3 : f.c * z3 + f.d $\neq$ 0) :
  let fzi := (f.a * zi + f.b) / (f.c * zi + f.d)
  (f z0 - f z1) / (f z0 - f z3) * ((f z2 - f z3) / (f z2 - f z1)) =
  (z0 - z1) / (z0 - z3) * ((z2 - z3) / (z2 - z1))
\end{lstlisting}


The full Lean proof is about 1400 lines. Its length reflects the explicit verification of cases that are usually left implicit. 
Our formal proof records the complete boundary behavior of the cross ratio on $\mathbb{C}_\infty$ and guarantees that no case involving infinity or a pole of the transformation has been overlooked.


\subsection{Permutation Generators}

Permuting the four arguments transforms the cross ratio by one of six elementary fractional expressions. More precisely, if
\[
\lambda = [z_0,z_1;z_2,z_3],
\]
then the possible values obtained by permuting the four points are
\begin{equation}
	\lambda,\quad
	1-\lambda,\quad
	\frac{1}{\lambda},\quad
	\frac{1}{1-\lambda},\quad
	\frac{\lambda-1}{\lambda},\quad
	\frac{\lambda}{\lambda-1}.
	\label{eq:6_values}
\end{equation}
Thus, although there are $24$ permutations of four points, they give rise to only six distinct values of the cross ratio. This classical fact reflects the existence of a normal subgroup of $S_4$ acting trivially on the cross ratio, so that the induced action factors through a group of order six.

In our formalization, we formalize two generating transformations. 
\begin{lstlisting}[language=lean,mathescape=true]
lemma cross_ratio_swap_12_complement 
  (z0 z1 z2 z3 : EComplex)
  (h_distinct : List.Pairwise ($\cdot \neq \cdot$) [z0, z1, z2, z3]) :
  cross_ratio z0 z2 z1 z3 = 1 - cross_ratio z0 z1 z2 z3
\end{lstlisting}

\begin{lstlisting}[language=lean,mathescape=true]
lemma cross_ratio_swap_13_inv 
  (z0 z1 z2 z3 : EComplex)
  (h_distinct : List.Pairwise ($\cdot \neq \cdot$) [z0, z1, z2, z3]) :
  cross_ratio z0 z3 z2 z1 = 1 / cross_ratio z0 z1 z2 z3
\end{lstlisting}

The first corresponds to exchanging the second and third arguments, and sends the cross ratio to its complement. The second corresponds to exchanging the second and fourth arguments, and sends the cross ratio to its inverse.

These two rules generate the full set of six values in \eqref{eq:6_values}.
Indeed, repeated application of the two transformations
\[
\lambda \mapsto 1-\lambda
\qquad\text{and}\qquad
\lambda \mapsto \frac{1}{\lambda}
\]
produces exactly these six expressions. The proofs follow the same pattern: case analysis on which arguments are $\infty$ and then algebraic simplification using the field laws.

\section{Discussion: Algebraic Formalization and AI Assistance}
\label{sec:discussion}

Our choice of a purely algebraic representation (the \texttt{Option} type for $\mathbb{C}_\infty$ and explicit matrix-like coefficients for M\"obius transformations) stands in contrast to the approach of homogeneous coordinates or quotient types. In our development, the heavy use of \texttt{field\_simp}, \texttt{ring}, and \texttt{norm\_cast} automates most algebraic manipulations, leaving only the structural case splits to be handled manually.

Nevertheless, the case splits are numerous: for \texttt{comp\_equivalent} we have 8 main branches, and for \texttt{cross\_ratio\_invariant} we have dozens of subcases. This is precisely the kind of repetitive, structurally predictable work that could be automated by a large language model or by a domain-specific tactic in Lean. We envision that future versions of this formalization could be generated partially by AI, with the human formalizer focusing on the core algebraic identities. Our work provides a sample of case-split proofs that could be used to train such models.

\section{Conclusion and Future Work}
\label{sec:conclusion}

We have presented a complete, machine-checked formalization of the extended complex plane, M\"obius transformations, and the cross ratio in Lean 4. Our development comprises approximately 6,000 lines of Lean 4 code (including definitions, lemmas, and proofs), with over 40 definitions and 150 lemmas or theorems. The key results are:
\begin{itemize}
	\item A simple model of $\mathbb{C}_\infty$ using \texttt{Option $\mathbb{C}$}.
	
	\item Linear fractional transformations as a group acting on $\mathbb{C}_\infty$, and and a realization of the symmetric group $S_4$ as a subgroup of the group of~LFT.
	
	\item A proof that a M\"obius transformation is uniquely determined by its action on three points, together with an explicit existence construction.
	
	\item A definition of the cross ratio covering points at infinity, a proof of its M\"obius invariance. 
\end{itemize}

This work lays a foundation for further formalization in complex analysis, conformal geometry, and physics. Immediate future directions include:
\begin{itemize}
	\item \textbf{Generalized circles (circlines)}: formalize the notion that M\"obius transformations map circles and lines to circles and lines, and provide a unified representation using Hermitian matrices.
	\item \textbf{Conformality}: prove that M\"obius transformations are conformal (angle-preserving) using either the complex derivative or the cross-ratio definition of angles. This will connect our work to differential geometry and complex analysis.
	\item \textbf{Applications in physics}: implement the isomorphism between $PSL(2,\mathbb{C})$ and the Lorentz group $SO^+(3,1)$ using Pauli matrices, formalizing the transformation of the celestial sphere in special relativity.
	\item \textbf{Tactic development}: create a tactic that automatically generates case-split proofs for cross-ratio invariance and other similar properties, leveraging the structural patterns observed in our manual proofs.
\end{itemize}
	
All code is available at \url{https://github.com/fubinyan/ComplexVariables/blob/main/ComplexAnalysis/Chapter2_2.lean}.

	
	
\end{document}